\documentclass[10pt]{amsart}
\usepackage{amsmath,amsthm,amssymb,url}
\include{xy}
\xyoption{all}
\SelectTips{cm}{}

\newcommand\Z{{\mathbb Z}}
\newcommand\Zpos{\Z_{\ge1}}

\newcommand\Znn{\Z_{\ge0}}

\newcommand\Q{{\mathbb Q}}

\newcommand\C{{\mathbb C}}

\newcommand\F{{\mathbb F}}
\newcommand\Fp{\F_{\!p}}

\newcommand\lset{\{\,}
\newcommand\rset{\,\}}
\newcommand\lra{\longrightarrow}

\newcommand\inv{^{-1}}
\newcommand\tr{\operatorname{tr}}

\newcommand\ie{{\em i.e.}}

\newcommand\ip[2]{\langle #1,#2\rangle} 

\newcommand\siX[1]{{\mathcal X}_{#1}} 

\newcommand\XtwoN{\siX2(N)}

\newcommand\SLgp{\operatorname{SL}}
\newcommand\SL[2]{\SLgp_{#1}(#2)}
\newcommand\SLtwoZ{\SL2\Z}

\newcommand\cmat[4]{\left[\begin{array}{cc}
{#1}&{#2}\\{#3}&{#4}\end{array}\right]}
\newcommand\smallmat[4]{\left[\begin{smallmatrix}
{#1}&{#2}\\{#3}&{#4}\end{smallmatrix}\right]}

\newcommand\paramodulargroup[1]{{\rm K}(#1)}
\newcommand\KN{\paramodulargroup N}

\newcommand\KNM{\paramodulargroup{NM}}
\newcommand\pvar{{\cmat\tau zz\omega}} 

\newcommand\fcJ[2]{c(#1;#2)} 

\newcommand\wtvar[2]{[#1]_{#2}}
\newcommand\wtk[1]{\wtvar{#1}k}

\newcommand\fc[2]{a(#1;#2)} 
\newcommand\e{{\rm e}} 

\newcommand\CFsnoweight{{\mathcal S}} 
\newcommand\CFswtgp[2]{\CFsnoweight_{#1}(#2)}

\newcommand\anyS{\CFsnoweight}
\newcommand\SkKN{\CFswtgp k\KN}

\newcommand\StwoKN{\CFswtgp2\KN}

\newcommand\StwoKlevel[1]{\CFswtgp2{\paramodulargroup{#1}}}
\newcommand\StwoKtwoN{\StwoKlevel{2N}}
\newcommand\SwtKlevel[2]{\CFswtgp{#1}{\paramodulargroup{#2}}}
\newcommand\SfourKN{\CFswtgp4\KN}

\newcommand\Jcusp[2]{{\rm J}_{#1,#2}^{\rm cusp}}

\newcommand\Jkmcusp{\Jcusp km}
\newcommand\JkjNcusp{\Jcusp k{jN}}

\newcommand\JtwoNcusp{\Jcusp2N}

\newcommand\Jwh[2]{{\rm J}_{#1,#2}^{\rm w.h.}}

\newcommand\JkNwh{\Jwh kN}

\newcommand\JzeroNwh{\Jwh0N}

\newcommand\Grit{\opratorname Grit}
\newcommand\GritJkNcusp[2]{\Grit(\Jcusp{#1}{#2})}
\newcommand\GritJtwoNcusp{\GritJkNcusp2N}

\def\Grit{\operatorname{Grit}}
\def\JRMJ{{\mathcal J}}

\def\JRMJdv{\JRMJ_d^v}
\def\JRMJdpv{\JRMJ_{d,p}^v}

\def\detmax{{\det_{\max}}}

\def\TD{\operatorname{TD}}

\newcommand\mymod{\text{ mod }}

\newcommand\BL{{\rm Borch}}

\theoremstyle{plain}
\newtheorem{theorem}{Theorem}[section]

\newtheorem{lemma}[theorem]{Lemma}

\begin{document}
\title[Paramodular eigenform constructions]
{Nonlift weight two paramodular eigenform constructions}

\author[C.~Poor]{Cris Poor}
\address{Department of Mathematics, Fordham University, Bronx, NY 10458 USA}
\email{poor@fordham.edu}

\author[J.~Shurman]{Jerry Shurman}
\address{Department of Mathematics, Reed College, Portland, OR 97202 USA}
\email{jerry@reed.edu}

\author[D.~Yuen]{David S.~Yuen}
\address{Department of Mathematics, University of Hawaii, Honolulu, HI 96822 USA}
\email{yuen@math.hawaii.edu}

\date{\today}

\begin{abstract}
We complete the construction of the nonlift weight two cusp
para\-modular Hecke eigenforms for prime levels $N<600$, which arise in
conformance with the paramodular conjecture of Brumer and Kramer.
\end{abstract}


\maketitle


\section{Introduction\label{sectionIntr}}

We use Borcherds products to construct all the paramodular nonlift
newforms that are suggested by the work of \cite{py15}.
That article and this one provide evidence for the paramodular
conjecture of A.~Brumer and K.~Kramer, initially stated in
\cite{brkr14}.  The paramodular conjecture is now modified in
section~8 of \cite{brkr18}, after F.~Calegari pointed out an oversight
in the conjecture's earlier statement.
Define an abelian variety $B/K$ to be a {\em QM abelian variety\/} or
{\em QM by~$D$ abelian variety\/}
if ${\rm End}_KB$ is an order in the non-split quaternion algebra $D/\Q$.
Define a cuspidal, nonlift Siegel paramodular newform~$f$ of degree~$2$,
weight~$2$ and level~$N$ with rational Hecke eigenvalues to be a
{\em suitable\/} paramodular form of level~$N$.
The paramodular cusp form space is denoted $\StwoKN$---the
subscript~$2$ indicates the weight, $\KN$ denotes the paramodular
group of degree~$2$ and level~$N$, and the degree is omitted from the
notation because all paramodular forms in this article have
degree~$2$.
Newforms on~$\KN$ are by definition Hecke eigenforms orthogonal to the
images of level-raising operators from paramodular forms of lower
levels \cite{robertsschmidt06,robertsschmidt07}.
The lift space in $\StwoKN$ is $\GritJtwoNcusp$, the Gritsenko
(additive) lift of the Jacobi cusp form space of weight~$2$ and
index~$N$.
The paramodular conjecture is:
\begin{quote}
{\em Let $\mathcal A_N$ be the set of isogeny classes of abelian
  surfaces $A/\Q$ of conductor~$N$ with ${\rm End}_\Q A=\Z$, let
  $\mathcal B_N$ be the set of isogeny classes of QM abelian fourfolds
  $B/\Q$ of conductor~$N^2$, and let $\mathcal P_N$ be the set of
  suitable paramodular forms of level~$N$, up to nonzero
  scaling. There is a bijection
  $\mathcal A_N\cup\mathcal B_N\longleftrightarrow\mathcal P_N$
  such that
  $$
  L(C,s,\text{\rm Hasse--Weil})=\begin{cases}
  L(f,s,{\rm spin})  &\text{if $C\in\mathcal A_N$},\\
  L(f,s,{\rm spin})^2&\text{if $C\in\mathcal B_N$}.
  \end{cases}
  $$
}
\end{quote}

In \cite{py15}, the first and third authors of this article studied
$\StwoKN$ for prime levels $N<600$, proving by algorithm that
$\StwoKN=\GritJtwoNcusp$ for all such~$N$ other than the
exceptional cases $N=277,349,353,389,461,523,587$, precisely the
primes~$N<600$ for which relevant abelian surfaces exist \cite{brkr14}.
Also, \cite{py15} shows that in the exceptional cases there is at most
one nonlift dimension, lying in the Fricke plus space, except that for
$N=587$ there is at most one Fricke plus space nonlift dimension and
at most one Fricke minus space nonlift dimension.  We refer to these
last two cases as $N=587^\pm$.  At each relevant level, granting that
the nonlift exists, the relevant newform~$f=f_N$ is known (see the
website \cite{yuen15a}), as are Euler factors of $L(f,s,{\rm spin})$
for small primes \cite{py15}.
Further, \cite{py15} shows that $\StwoKlevel{277}$ does
contains a nonlift dimension $\C f_{277}$, by constructing it.
In \cite{gpy16}, V.~Gritsenko and the first and third authors of this
article constructed a nonlift in $\StwoKlevel{587}^-$, using a
Borcherds product.
A.~Brumer and A.~Pacetti and G.~Tornar\'ia and J.~Voight and the first
and third authors of this article have shown the equality of
$L(f_N,s,{\rm spin})$ and $L(A_N,s,\text{Hasse--Weil})$,
for $N=277,353,587^-$,
citing this article for the existence of the nonlift for~$N=353$
\cite{bpptvy18,yuen15}.
We call $N=277,349,353,389,461,523,587^\pm$ the {\em outstanding levels\/},
even though some of them are already addressed.
This article describes nonlift constructions for all of them,
completing \cite{py15}.
All but two outstanding levels have nonlift Borcherds products;
the more interesting cases are the two exceptions, $N=461,587^+$.
Thus our main result, also including the squarefree composite levels
through~$300$, studied in \cite{MR3713095}, is

\begin{theorem}
The following dimensions are established.
$$
\begin{array}{|c||ccc|}
\hline
N&\dim\JtwoNcusp&\dim\StwoKN^+&\dim\StwoKN^-\\
\hline\hline
249&5&6&0\\
277&10&11&0\\
295&6&7&0\\
349&11&12&0\\
353&11&12&0\\
389&11&12&0\\
461&12&13&0\\
523&17&18&0\\
587&18&19&1\\
\hline
\end{array}
$$
\end{theorem}

More details of the nonlift computations are given at the
website \cite{yuen18}.
In consequence of this theorem, the provisional nonlift eigenform
formulas in \cite{py15} are now established.  Expressions for the
nonlift eigenforms as linear combinations of our constructed nonlifts
and Gritsenko lifts will appear in a later expanded version of this
article.

\section{Background\label{sectionBG}}

The background section of \cite{MR3713095} introduces terminology and
notation for paramodular forms and for Fricke and Atkin--Lehner involutions.
Here we add some specifics that bear on this article.
Let $\SkKN$ denote the space of weight~$k$, level~$N$ paramodular cusp
forms.  For $k\ge3$ and squarefree~$N$, a formula for $\dim\SkKN$ is
given in \cite{ik17}, but no such formula is known for~$k=2$.  Every
paramodular cusp form of weight~$k$ and level~$N$ has a Fourier expansion
$$
f(\Omega)=\sum_{t\in\XtwoN}\fc tf\,\e(\ip t\Omega)
$$
where
$\XtwoN=\lset\smallmat n{r/2}{r/2}{mN}:n,m\in\Znn,r\in\Z,4nmN-r^2>0\rset$
and $\e(z)=e^{2\pi iz}$ and $\ip t\Omega=\tr(t\Omega)$.
The Fourier--Jacobi expansion of such a paramodular cusp form is
$$
f(\Omega)=\sum_{m\ge1}\phi_m(f)(\tau,z)\xi^{mN},
\quad \Omega=\pvar,\ \xi=\e(\omega)
$$
with Jacobi coefficients
$$
\phi_m(f)(\tau,z)
=\sum_{t=\smallmat n{r/2}{r/2}{mN}\in\XtwoN}\fc tfq^n\zeta^r,
\quad q=\e(\tau),\ \zeta=\e(z).
$$
Here the coefficient $\fc tf$ is also the coefficient $\fcJ{n,r}{\phi_m}$.
That is,
$$
\phi_m(f)(\tau,z)
=\sum_{n,r:4nmN-r^2>0}\fcJ{n,r}{\phi_m}q^n\zeta^r.
$$
Each Jacobi coefficient $\phi_m(f)$ lies in the space
$\Jcusp k{mN}$ of weight~$k$, index~$mN$ Jacobi cusp forms,
whose dimension is known \cite{ez85,sz89}.  For the theory of Jacobi forms,
see \cite{ez85,gn98,sz89}.

Let $\anyS$ denote any Fricke eigenspace or Atkin--Lehner eigenspace
of $\SkKN$.  To describe our computational method for studying~$\anyS$
in weight $k=2$, introduce for any positive integer~$d$ the notation
\begin{alignat*}2
\anyS(d)&=&&\lset\text{$\anyS$-elements with vanishing first $d-1$
  Jacobi coefficients}\rset,\\
\anyS[d]&=&&\lset\text{$\anyS$-elements truncated to the first $d$
  Jacobi coefficients}\rset.
\end{alignat*}
Thus there is an exact sequence
$$
0\lra\anyS(d+1)\lra\anyS\lra\anyS[d]\lra0.
$$
Our prototype methodology to study~$\anyS$ in weight~$2$ is as follows.
\begin{enumerate}
\item Show for some nonnegative integer~$d$ that $\anyS(d+1)=0$, i.e.,
  every element of~$\anyS$ is determined by its first~$d$ Jacobi
  coefficients.
\item Find an efficient superspace~$\JRMJ$ of $\anyS[d]$.  We are
  confident that $\JRMJ$ is exactly $\anyS[d]$, so we believe that the
  bound $\dim\anyS\le\dim\JRMJ$ is an equality, and we can identify
  elements of~$\anyS$ uniquely as linear combinations of our basis
  of~$\JRMJ$.
\item Show that $\dim\anyS=\dim\JRMJ$, and span~$\anyS$ in the
  process, by constructing that many independent elements in~$\anyS$.
\end{enumerate}
We use variations of these ideas when convenient, such as identifying
an element $f[d]$ in~$\JRMJ$ even when this determines~$f$ only up to 
$\anyS(d+1)$, and when $\JRMJ$ may well contain $\anyS[d]$ properly.

The methods described in the previous paragraph are facilitated by
working in some prime characteristic~$p$, rather than in
characteristic~$0$; as will be explained below, this doesn't lose any
dimensions.
We describe the transition from characteristic~$0$ to characteristic~$p$.
Continue to let $\anyS$ denote any Fricke eigenspace or Atkin--Lehner
eigenspace of $\SkKN$, and let $\Jkmcusp$ denote the space of
weight~$k$, index~$m$ Jacobi cusp forms for any~$m$.
Using the Jacobi expansions of paramodular cusp forms, we view two
maps as containments for simplicity,
$$
\anyS\subset\bigoplus_{j=1}^\infty\JkjNcusp\subset\C^\infty.
$$
The latter containment entails some chosen ordering  of the index sets
for Fourier expansions of Jacobi forms.
For any prime~$p$, let a subscript~$p$ denote the map from subsets
of~$\C^\infty$ to subsets of~$\Fp^\infty$ that reduces integral
elements modulo~$p$,
$$
V_p = (V\cap\Z^\infty)\mymod p,\quad V\subset\C^\infty.
$$
(So $(av+\tilde v)_p=a_pv_p+\tilde v_p$ for $a\in\Z$ and $v,\tilde
v\in\Z^\infty$, where $a_p$ is the image of~$a$ in~$\Fp$.)
Thus
$$
\anyS_p\subset\bigoplus_{j=1}^\infty(\JkjNcusp)_p\subset\Fp^\infty.
$$
The spaces  $\anyS$ and $\Jkmcusp$ have bases in~$\Z^\infty$ by
results of \cite{shim75} and \cite{ez85}, respectively.
Therefore $\dim_{\Fp}\anyS_p=\dim_\C\anyS$
and $\dim_{\Fp}(\Jkmcusp)_p=\dim_\C\Jkmcusp$.
Extending the notation from above, again for any positive integer~$d$,
\begin{alignat*}2
\anyS_p(d)&=&&\lset\text{elements of $\anyS_p$ with vanishing first
  $d-1$ Jacobi coefficients in~$\Fp$}\rset,\\
\anyS[d]_p&=&&\lset\text{reductions modulo $p$ of the integral
  elements of $\anyS[d]$}\rset,
\end{alignat*}
and we will show below that in consequence of an integral basis of~$\anyS$,
\begin{align*}
\dim_\C\anyS(d)&\le\dim_{\Fp}\anyS_p(d),\\
\dim_\C\anyS[d]&=\dim_{\Fp}\anyS[d]_p.
\end{align*}
Thus, establishing that $\anyS_p(d)=0$ establishes that $\anyS(d)=0$,
and bounding $\dim\anyS[d]_p$ gives the same bound of~$\dim\anyS[d]$.
Here the spaces $\anyS_p(d)$ and $\anyS[d]_p$, which we use, have
different definitions than the spaces $\anyS(d)_p$ and~$\anyS_p[d]$,
which we don't.

For the first part of our three-part method in weight~$2$, the
methods of \cite{MR3713095} let us attempt to show for a given~$d$
that $\anyS(d)=0$ or $\anyS_p(d)=0$, where $\anyS$ is either Fricke
eigenspace of~$\StwoKN$, or even all of~$\StwoKN$.
These methods require us to span enough of the weight~$4$ space
$\SfourKN$, namely a subspace of codimension less than $\dim\Jcusp2N$.
As already noted, we know $\dim\SfourKN$ from \cite{ik17} if $N$ is
squarefree, and we know $\dim\Jcusp2N$ from \cite{ez85,sz89}.

For the second part of our method, the Jacobi Restriction method
\cite{bpy16,ipy13,MR3713095} gives an efficient superspace~$\JRMJ$ of
$\anyS[d]$ or~$\anyS[d]_p$, and thus a good upper bound of
$\dim\anyS[d]$ or $\dim\anyS[d]_p$.
Specifically, let $v$ denote either a Fricke eigenvalue or a vector of
Atkin--Lehner eigenvalues, and let $\anyS^v$ denote the corresponding
Fricke or Atkin--Lehner eigenspace of~$\SkKN$.
Jacobi Restriction runs with $d$ as a parameter, returning a basis of
a finite-dimensional complex vector space $\JRMJdv$ of truncated
formal Fourier--Jacobi expansions such that
$$
\anyS^v[d]\subset\JRMJdv\subset\bigoplus_{j=1}^d\JkjNcusp.
$$
An additional parameter, $\detmax$, bounds the Fourier coefficient
index determinants in the calculation, thus making the method an
algorithm; this parameter must be chosen so that each space $\JkjNcusp$
with~$j\le d$ is determined by the Fourier coefficients whose indices
satisfy the determinant bound.  For simplicity we suppress $\detmax$
from the notation $\JRMJdv$, along with the weight~$k$ and the level~$N$.
Jacobi restriction gives a basis of~$\JRMJdv$, each basis element
represented as a finite collection of data $a(n,r,m)$ where $1\le m\le d$
and $0<nmN-r^2/4\le\detmax$; the truncation of any paramodular cusp form
$f\in\SkKN$ to its first $d$ Jacobi coefficients is determined by its
Fourier coefficients $\fc{\smallmat n{r/2}{r/2}{mN}}f$ for such $n,r,m$.
Jacobi Restriction also can be run modulo~$p$, in which case it
returns a basis of a finite dimensional vector space $\JRMJdpv$
over~$\Fp$ such that
$$
\anyS^v[d]_p\subset\JRMJdpv\subset\bigoplus_{j=1}^d(\JkjNcusp)_p.
$$

The third part of our method is much more situational, involving a
range of context-dependent constructions to create paramodular cusp
forms in the desired space.  For this reason, we attempt no systematic
description here.

We show the relations $\dim_\C\anyS(d)\le\dim_{\Fp}\anyS_p(d)$ 
and $\dim_\C\anyS[d]=\dim_{\Fp}\anyS[d]_p$ for any $d\ge1$ and any
prime~$p$, as stated above.  The relation $\dim_\C\anyS=\dim_{\Fp}\anyS_p$
has already been noted, so the inequality needs to be shown only for~$d\ge2$.
The key is that reducing a saturated lattice modulo a prime doesn't
decrease its dimension.

\begin{lemma}
Let $M\subset\Z^\infty$ be a $\Z$-module of finite rank~$d$.
Suppose that $M$ is saturated, meaning that the general containment
$(\Q\otimes M)\cap\Z^\infty\supset M$ is equality,
$$
(\Q\otimes M)\cap\Z^\infty= M.
$$
Let $p$ be prime, and let $M_p\subset\Fp^\infty$ denote the image
of~$M$ under reduction modulo~$p$, a vector space over~$\Fp$.
Then $\dim M_p=d$.
\end{lemma}

\begin{proof}
Consider a basis $\lset v_1,\dotsc,v_d\rset$ of~$M$.
The corresponding set of reductions,
$\lset v_{1,p},\dotsc,v_{d,p}\rset$, spans~$M_p$.
Consider any linear relation over~$\Fp$ among the reductions,
$\sum_ic_{i,p}v_{i,p}=0$.  The relation gives a congruence
$\sum_ic_iv_i=0\mymod p$ in~$\Z^\infty$, and so the vector
$\sum_i(c_i/p)v_i$ lies in $(\Q\otimes M)\cap\Z^\infty$, which by
hypothesis is~$M$, which is $\bigoplus_i\Z v_i$.  So each $c_i$ lies
in~$p\Z$, and the linear combination over~$\Fp$ is trivial.
Thus $\lset v_{1,p},\dotsc,v_{d,p}\rset$ is a basis of~$M_p$.
\end{proof}

Now, recall that the complex vector space~$\anyS$ has an integral basis.
Consequently, the projection map from $\anyS$ to~$\anyS[d]$ is defined
over~$\Z$, and hence so are its kernel and image.  That is,
$\anyS(d+1)$ and $\anyS[d]$ have integral bases.
Consequently the lattices $M(d+1)=\anyS(d+1)\cap\Z^\infty$ and
$M[d]=\anyS[d]\cap\Z^\infty$ have respective ranks
$\dim_\C\anyS(d+1)$ and $\dim_\C\anyS[d]$, and they are saturated.
By the lemma, the dimensions of their reductions $M(d+1)_p$ and
$M[d]_p$ as vector spaces over~$\Fp$ match their ranks.
Because $M(d+1)_p$ lies in~$\anyS_p(d+1)$, this gives the first desired
relation, $\dim_\C\anyS(d)\le\dim_{\Fp}\anyS_p(d)$, with $d+1$ in place
of~$d$, i.e., for all~$d\ge2$.
And because $M[d]_p=\anyS[d]_p$, we have the second desired relation
as well, $\dim_\C\anyS[d]=\dim_{\Fp}\anyS[d]_p$.


We quickly review Atkin--Lehner involutions.
Let~$N$ be a positive integer, and let~$c$ be a positive divisor
of~$N$ such that $\gcd(c,N/c)=1$.  For any four integers $\alpha$,
$\beta$, $\gamma$, $\delta$ such that $\alpha\delta c-\beta\gamma N/c=1$, 
define a $c$-Atkin--Lehner matrix
$$
\alpha_c=\frac1{\sqrt c}\cmat{\alpha c}\beta{\gamma N}{\delta c}.
$$
Especially, for $c=1$ we may take $\alpha,\delta=1$ and
$\beta,\gamma=0$ to get the identity matrix, and for $c=N$ we may take
$\alpha,\delta=0$ and $\beta,\gamma=\mp1$ to get the usual Fricke
involution matrix $\frac1{\sqrt N}\smallmat0{-1}N{\phantom{-}0}$.
Let $\Gamma_0(N)$ denote the level~$N$ Hecke subgroup of~$\SLtwoZ$.
By quick calculations, the inverse of any~$\alpha_c$ is
another~$\tilde\alpha_c$, any product $\tilde\alpha_c\alpha_c$ lies in
$\Gamma_0(N)$ and so the set of all~$\tilde\alpha_c$ lies in the
coset $\Gamma_0(N)\alpha_c$, and this coset also lies in the set of
all~$\tilde\alpha_c$, making them equal.  Consequently, $\alpha_c$
squares into~$\Gamma_0(N)$ and normalizes~$\Gamma_0(N)$.
Now let $\mu_c=\alpha_c^*\boxplus\alpha_c$, where the asterisk and the
box-plus are the matrix transpose-inverse and direct sum operators.
The inverse of any~$\mu_c$ is another $\tilde\mu_c$, and any product
$\tilde\mu_c\mu_c$ lies in~$\KN$, so that the set of all~$\tilde\mu_c$
lies in the coset $\KN\mu_c$ and they all give the same action on
paramodular forms, but now the containment is proper.
Again $\mu_c$ squares into~$\KN$, and a blockwise check shows that
$\mu_c$ normalizes~$\KN$.
If $c$ divides~$N$ and $\gcd(c,NM/c)=1$ then any level~$NM$
matrix~$\alpha_c$ is also a level~$N$ matrix~$\alpha_c$, and so the
same is true for the corresponding~$\mu_c$.

The previous paragraph shows that for any prime divisor~$p$ of~$N$
such that $\gcd(p,NM/p)=1$, any $p$-Atkin--Lehner eigenform~$f$ in
$\SwtKlevel k{NM}$ traces down to a $p$-Atkin--Lehner eigenform in
$\SwtKlevel kN$.  Indeed, if $\KNM\KN=\bigsqcup_i\KNM h_i$ then the
traced down image of~$f$ is $\TD f=\sum_if\wtk{h_i}$, and
taking a $p$-Atkin--Lehner matrix~$\mu_p$ at level~$N$ that is also a
$p$-Atkin--Lehner matrix at level~$NM$ gives, because
$f\wtk{\mu_p}=\epsilon_pf$ and $\KNM\KN=\bigsqcup_i\KNM\mu_p\inv h_i\mu_p$,
$$
(\TD f)\wtk{\mu_p}
=\sum_if\wtk{\mu_p\mu_p\inv h_i\mu_p}
=\epsilon_p\TD f.
$$
In this article we take $N$ an odd prime and~$M=2$, and trace down a
level~$2N$ Atkin--Lehner eigenform having $2$-eigenvalue~$-1$ and
$N$-eigenvalue~$1$, obtaining a level~$N$ Fricke plus form.
This is our route to constructing a nonlift in a space that has no
nonlift Borcherds products.

\section{Jacobi form constructions for Borcherds
  products\label{sectionwt0Jf}}

Borcherds products are a source of nonlift paramodular forms,
although some of them are lifts.
The theory of Borcherds products for paramodular forms is
given by V.~Gritsenko and V.~Nikulin in~\cite{gn98}.
A Borcherds product takes the form $f=\BL(\psi)$ where
$\psi\in\JzeroNwh$ is a weight~$0$ weakly holomorphic Jacobi form.
Theorem~3.3 of \cite{gpy15}, which in turn is quoted from
\cite{gn98,grit12} and relies on the work of R.~Borcherds,
gives sufficient conditions for a Borcherds product to be a
paramodular Fricke eigenform; see also section~7 of \cite{MR3713095}.

One source of weight~$0$ weakly holomorphic Jacobi forms is the containment
$\Jcusp{12r}N/\Delta^r\subset\JzeroNwh$ for~$r\in\Zpos$, with $\Delta$
the discriminant function from elliptic modular forms, a weight~$12$
cusp form.  In particular, Jacobi forms $\psi\in\Jcusp{12}N/\Delta$
have nonlift Borcherds products in~$\StwoKN$ for all outstanding
levels except for $N=461,587^+$.
Jacobi forms created by this method tend to be long linear
combinations of a basis of $\Jcusp{12}N/\Delta$, and their
coefficients tend to be rational numbers with large numerators and
denominators.

A second source of weight~$0$ weakly holomorphic Jacobi forms~$\psi$
is a method that we call the {\em in\-fla\-tion method\/}.  It builds $\psi$
from a combination of two sums, one that involves raising operators
and a second that involves inflations (to be explained below),
$$
\psi=\sum_{i=1}^m\alpha_i(\phi_{1,i}|V_2)/\phi_{1,i}
+\sum_{j=1}^n\beta_j\Theta_j/\phi_{2,j}.
$$
Here the $\phi_{1,i}$ and the $\phi_{2,j}$ and the $\Theta_j$ are
basic theta blocks, by which we mean theta blocks that are weakly
holomorphic Jacobi forms of integral weight and level, and $V_2$ is
the index-raising operator of \cite{ez85}, page~$41$.  Each theta
block~$\phi_{1,i}$ lies in $\JkNwh$, has $q$-vanishing
order~$\nu_{1,i}$, and has baby theta block $b_{1,i}(\zeta)$ such that
$b_{1,i}(\zeta)\mid b_{1,i}(\zeta^2)$ \cite{psyallbp};
each theta block~$\phi_{2,j}$ lies in $\Jwh k{r_jN}$ for some~$r_j$
and has $q$-vanishing order~$\nu_{2,j}$ and baby theta block
$b_{2,j}(\zeta)$, and each theta block $\Theta_j$ lies in $\Jwh
k{(r_j+1)N}$ and has $q$-vanishing order $\tilde\nu_j$ and baby theta
block $\tilde b_j(\zeta)$, and $b_{2,j}(\zeta)\mid\tilde b_j(\zeta)$.
These conditions make each term in the previous display an element of
$\JzeroNwh$ \cite{gpy15}.
We call $\Theta_j$ an {\em inflation} of~$\phi_{2,j}$ \cite{gpy16}.
Some special cases of the inflation method are as follows.
\begin{itemize}
\item{\em Case~1.\/} $n=0$, so that
  $\psi=\sum_{i=1}^m\alpha_i(\phi_i|V_2)/\phi_i$, and furthermore
  all~$\nu_i$ are~$1$.
\item{\em Case~2.\/} $(m,n)=(1,1)$ with one $\phi$,
  such that $\nu(\phi)\in\{1,2\}$, and
  $\psi=(-1)^{\nu(\phi)}(\phi|V_2)/\phi+\beta\Theta/\phi$.  Here
  $\beta$ is usually~$\pm1$.  The first term of~$\psi$ determines $r=1$
  in the second.
\item{\em Case~3.\/} $(m,n)=(0,1)$, so that
  $\psi=\Theta/\phi$.  Usually $(\nu,\tilde\nu)=(1,2)$
  or $(\nu,\tilde\nu)=(2,2)$.
\end{itemize}
For each of the outstanding levels except for $N=461,587^+$, a Jacobi
form~$\psi$ that gives rise to a nonlift Borcherds product can be
obtained by various cases, rather than taking $\psi\in\Jcusp{12}N/\Delta$.
A nonlift Borcherds product of level $277$ can be constructed by the
Case~1 method with $m=3$, and also it can be constructed by
the Case~2 method with $(\nu,\tilde\nu)=(1,2)$ and~$\beta=-1$;
further, a different nonlift Borcherds product of level~$277$, this
one having its leading Jacobi coefficient in $\Jcusp2{554}$ rather than
$\Jcusp2{277}$, can be constructed by the Case~3 method, with
$(r,\nu,\tilde\nu)=(2,2,2)$.
For level $349$, the Case~1 method works with $m=22$, and the
resulting nonlift Borcherds product has a leading theta block with
denominator.  
For level $353$, the Case~1 method works with $m=3$, and Case~2 with
$(\nu,\tilde\nu)=(1,2)$ and $\beta=-1$ constructs the same nonlift
Borcherds product.
For level $389$, Case~2 with $(\nu,\tilde\nu)=(1,2)$ and $\beta=1$
works.
For level $523$, Case~1 works with $m=52$, and the resulting nonlift
Borcherds product has a leading theta block with denominator.  
For level $587^-$, Case~2 with $(\nu,\tilde\nu)=(2,2)$ and $\beta=-1$
works \cite{gpy16}.

We mention briefly that in \cite{MR3713095}, the authors of this
article studied $\StwoKN$ for squarefree composite levels~$N<300$,
proving by algorithm that $\StwoKN=\GritJtwoNcusp$ for all such~$N$
other than the exceptional cases $N=249,295$, where there is at most
one nonlift dimension.  These are the two odd squarefree composite
values $N<300$ for which relevant abelian surfaces exist, and at these
two levels the one known isogeny class contains Jacobians of
hyperelliptic curves.
For these two levels, we constructed the putative nonlift eigenform
by the Case~3 method, with $(r,\nu,\tilde\nu)=(2,2,2)$, as
was done for~$N=277$.


\section{Levels $461$ and $587$\label{sectionwts461587}}

For levels $N=461,587$, the methods of \cite{psyallbp} show that
there is no nonlift Borcherds product in $\StwoKN^+$.  To construct
nonlifts in these two cases, we begin by constructing Borcherds
products in $\StwoKlevel{2N}^-$.  Again by the methods of \cite{psyallbp},
there is one Borcherds product in $\StwoKlevel{922}^-$ and three in
$\StwoKlevel{1174}^-$.  These Borcherds products have integral Fourier
coefficients, some of which we can compute.
Polarizing the Borcherds products creates Atkin--Lehner eigenforms
$f_{2N}^{-,+}$ in $\StwoKtwoN^{-,+}$ (\ie, their $\mu_2,\mu_N$-eigenvalues
are $-1,1$) again having integral Fourier coefficients, and the
eigenforms trace down to $\StwoKN^+$, the tracing down again preserving
the integrality of the Fourier coefficients; Atkin--Lehner involutions~$\mu_c$
and tracing down are described in \cite{MR3713095}, and they were
briefly touched on further in section~\ref{sectionIntr} here.
However, we don't have enough Fourier coefficients of our eigenforms to
trace them down, so we need to produce more coefficients first.  We do
so in a finite characteristic~$p$.
Level~$2N$ Jacobi Restriction with $p=12347$, with $d=5$, and with
$\detmax=2305$ for level~$922$ or $\detmax=3522$ for level~$1174$,
gives a one-dimensional superspace $\JRMJ_{5,p}^{-,+}$ of
$\StwoKlevel{2N}^{-,+}[5]_p$ in each case.  Thus any one nonzero
integral Fourier coefficient modulo~$p$ of some $f^{-,+}_{2N}$
shows that $\Fp f^{-,+}_{2N}[5]_p=\JRMJ_{5,p}^{-,+}$, and this determines
the Fourier coefficients of~$f^{-,+}_{2N}[5]_p$ in~$\Fp$ at the indices where
we have Fourier coefficients of the $\JRMJ_{5,p}^{-,+}$ basis element.
We refer to this process as {\em prolonging\/}~$(f^{-,+}_{2N})_p$.
Because $f^{-,+}_{2N}$ is an Atkin--Lehner eigenform, we can obtain
yet more of its Fourier coefficients modulo~$p$ by using the relations
$\fc{\alpha t\alpha'}{f^{-,+}_{2N}}=\epsilon\fc t{f^{-,+}_{2N}}$ where
$\alpha\in\{\alpha_2,\alpha_N\}$ is a $2\times2$ Atkin--Lehner
involution matrix and~$\alpha'$ is its transpose, and $\epsilon$ is
the corresponding eigenvalue.  We refer to this process as
{\em infilling\/}~$(f^{-,+}_{2N})_p$.  Prolonging and infilling our
initial fragment of~$(f^{-,+}_{2N})_p$ can give us enough information
to demonstrate that the desired nonlift exists at level~$N$.  Let
``$\operatorname{TD}$'' denote the trace down operator in
characteristic~$0$ and in characteristic~$p$.
For $N=461$, the available Fourier coefficients of $(f_{922}^{-,+})_p$
from prolonging and infilling lead to $255$ coefficients of its traced
down image $\TD((f_{922}^{-,+})_p)$.  These are more than enough to
show that the latter is linearly independent of the Gritsenko lifts
modulo~$p$.  Because $\TD((f_{922}^{-,+})_p)=(\TD f^{-,+}_{922})_p$,
it follows that $\TD f^{-,+}_{922}$ is a nonlift in~$\StwoKlevel{461}^+$.
Similarly, for $N=587$ we have $271$ Fourier coefficients of the
traced down image of the~$(f^{-,+}_{1174})_p$ that arises from the
second (or third) of the three nonlift Borcherds product in
$\StwoKlevel{1174}^-$, again plenty to determine that
$\TD f^{-,+}_{1174}$ is a nonlift in~$\StwoKlevel{587}^+$.
The ideas of this paragraph generalize beyond an odd prime level~$N$
and a one-dimensional space $\JRMJ_{d,p}^{-,+}$, but here we described
them only as needed for the situation at hand.

\bibliographystyle{plain}
\bibliography{inflation}

\end{document}